\documentclass[12pt]{article}
\usepackage{amssymb,comment}
\usepackage{comment}
\usepackage{amsmath,amsthm}
\usepackage{graphicx}
\usepackage{color}

\newcounter{foo}
\usepackage{titling}

\usepackage{url}

\def\Fq{{\mathbb{F}_q}}

\newcommand{\PG}{\mathrm{PG}}

\newcommand{\cG}{{\cal G}}

\newcommand{\cL}{{\cal L}}
\newcommand{\cM}{{\cal M}}
\newcommand{\cP}{{\cal P}}

\newcommand{\cT}{{\cal T}}
\newcommand{\cU}{{\cal U}}

\usepackage[dvipsnames]{xcolor}
\usepackage{hyperref}
\hypersetup{colorlinks = true, citecolor = NavyBlue, linkcolor = NavyBlue, urlcolor = NavyBlue}
\usepackage[backend=biber,giveninits=true,doi=false,isbn=false,url=false,eprint=false,maxbibnames=99,sortcites]{biblatex}
\usepackage{cleveref}

\newcommand{\red}[1]{\textcolor{red}{#1}}
\newfont{\blb}{msbm10 scaled\magstep1}
\newfont{\comp}{cmr12 scaled\magstep1}
\newfont{\compb}{cmr10 scaled\magstep2}
\newfont{\sbb}{cmssbx10 scaled\magstep3}
\newfont{\sbbb}{cmssbx10 scaled\magstep5}
\newfont{\sbs}{cmssbx10 scaled\magstep1}
\newtheorem{theorem}{Theorem}
\newtheorem{lemma}{Lemma}
\newtheorem{claim}[subsection]{Claim}

\newtheorem{corollary}{Corollary}
\newtheorem{proposition}{Proposition}
\newtheorem{observation}{Observation}

\newtheorem{conjecture}[foo]{Conjecture}

\theoremstyle{definition}
\newtheorem{definition}{Definition}
\newtheorem{remark}{Remark}

\makeatletter
\newtheorem*{rep@theorem}{\rep@title}
\newcommand{\newreptheorem}[2]{%
	\newenvironment{rep#1}[1]{%
		\def\rep@title{#2 \ref{##1}}%
		\begin{rep@theorem}}%
		{\end{rep@theorem}}}
\makeatother

\newreptheorem{theorem}{Theorem}
\newreptheorem{corollary}{Corollary}

\Crefname{owntheorem}{Theorem}{Theorems}
\Crefname{owncorollary}{Corollary}{Corollaries}
\Crefname{repcorollary}{Corollary}{Corollaries}
\Crefname{theorem}{Theorem}{Theorems}

\renewbibmacro{in:}{} %no "in"+ journal
\DeclareFieldFormat{journaltitle}{{#1},} % italic journal title with comma
\DeclareFieldFormat[inbook,thesis]{title}{\mkbibemph{#1}\addperiod} % italic title with period
\DeclareFieldFormat[article]{title}{\mkbibemph{#1}} % title of journal article is printed as normal text
\title{Improved bounds for the minimum degree of minimal multicolor Ramsey graphs}

\makeatletter
\renewcommand\@date{{%
		\vspace{-\baselineskip}%
		\large\centering
		\begin{tabular}{@{}c@{}}
			Yamaan Attwa\thanks{Institue for Mathematics, Freie Universit\"at Berlin, Germany. Funded by the Deutsche Forschungsgemeinschaft (DFG, German Research Foundation) under Germany´s Excellence Strategy – The Berlin Mathematics Research Center MATH+ (EXC-2046/1, project ID: 390685689).} 
		\end{tabular}
		\quad \begin{tabular}{@{}c@{}}
			Sam Mattheus\thanks{Department of Mathematics, Vrije Universiteit Brussel, Pleinlaan 2, 1050 Brussels, Belgium. E-mail: sam.mattheus@vub.be. Research supported by postdoctoral fellowship 1267923N from the Research Foundation Flanders (FWO).}   \end{tabular}%
		\quad   \begin{tabular}{@{}c@{}}
			Tibor Szabó\thanks{Institute  for Mathematics, Freie Universit\"at Berlin, Germany.}
		\end{tabular}
		\quad  \begin{tabular}{@{}c@{}}
			Jacques Verstraete\thanks{ Department of Mathematics, University of California, San Diego, 9500 Gilman Drive, La Jolla, CA 92093-0112, USA. E-mail: jacques@ucsd.edu. Research supported by the National Science Foundation FRG Award DMS-1952786 and NSF Award DMS-2347832.} 
		\end{tabular}
		
		\bigskip
		\bigskip
		\today

}}
\makeatother

\begin{document}
	\maketitle
	\begin{abstract}
		%Let $H$ be any graph. A graph $G$ is {\em $r$-Ramsey-minimal} for $H$ if every $r$-coloring of the edges of
		%$H$ gives a monochromatic copy of $H$ but no proper subgraph of $G$ has this property.
		%Let $s_r(H)$ denote the smallest minimum degree of any $r$-Ramsey minimal graph for $H$.
		%We improve various previous upper bounds 
		%%due to Fox, Grinshpun, Liebenau, Person and Szabo, who showed $s_r(K_k) = O_k(r^2(\log r)^{8k^2})$, Bishnoi, Bamberg and Lesgourgues, who showed $s_r(K_k) = O(k^5r^{5/2})$, and Bishnoi and Lesgourgues, who showed $s_r(K_k) = O(k^3r^3(\log r)^3)$
		%by showing for some absolute constant $c$ that for $k, r \geq 3$,  $s_r(K_{k+1}) \leq 2^{200(k+1)} r^2 \log r$ and also 
		%\[s_r(K_{k+1}) \leq (rk)^{2 + \frac{c}{k}}[(\log k)^2 + (\log k)^{1 + \frac{c}{k}}(\log r)^{1 + \frac{c}{k}}].\]
		%This implies for $k = \Omega(\log r)$ and $k = O(1)$ that $s_r(K_{k+1}) = O((rk\log k)^2 \log r)$.
		%This is close to best possible since for all $k,r \geq 3$, $s_r(K_{k+1}) \geq k^2$ and 
		%for $k = O(1)$, $s_r(K_{k+1}) = \Omega(r^2 \log r/\log \log r)$.
		
		We provide two novel constructions of $r$ edge-disjoint $K_{k+1}$-free graphs on the same vertex set, each of which has the property that every small induced subgraph contains a complete graph on $k$ vertices. The main novelty of our argument is the combination of an algebraic and a probabilistic coloring scheme, which utilizes the beneficial algebraic and combinatorial properties of the Hermitian unital. These constructions improve on a number of upper bounds on the smallest possible minimum degree of minimal $r$-color Ramsey graphs for the clique $K_{k+1}$ when $r\geq c\frac{k}{\log^2 k}$ and $k$ is large enough.  
		%We also resolve a question of Tran regarding the colored semisaturation of complete graphs. 
	\end{abstract}
	\section{Introduction}
	%\subsection{}
	%In 1976, Burr, Erd\H{o}s and Lov\'asz~\cite{burr1976graphs} initiated a systematic study of various properties of Ramsey graphs. TSz: Are they really studying PROPERTIES or just EXTREMAL VALUES OF PARAMETERS?
	We say that a graph $G$ is {\em $r$-Ramsey} for a graph $H$, denoted by $G \to (H)_r$ if every $r$-colouring of the edges of $G$ contains a monochromatic copy of $H$.  A graph $G$ is called {\em $r$-Ramsey-minimal for $H$} if it is $r$-Ramsey for $H$, but no proper subgraph of it is. The set of all $r$-Ramsey-minimal graphs for $H$ is denoted by $\cM_r(H)$. The classical Ramsey number $R_r(H)$, one of the %best
	most well-studied parameters in Combinatorics, is then the smallest number of vertices of a graph in $\cM_r(H)$. 
	Following the pioneering work of Folkman~\cite{folkman1970graphs} on the smallest clique number of Ramsey graphs for the clique, Burr, Erd\H{o}s and Lov\'asz~\cite{burr1976graphs} in 1976 initiated the systematic study of the extremal behaviour of several other graph parameters. 
	%In~\cite{burr1976graphs} Burr, Erd\H{o}s, and Lov\'asz investigated maximum/mainimum values of different parameters for graphs in $\cM_r(H)$, 
	In their seminal paper they investigated the chromatic number, the maximum and the minimum degree, and the connectivity.
	%Since their work, many researchers have studied various properties of this class and the graphs it contains. TSz: Again are they really studying PROPERTIES?
	Subsequently, their work inspired many further investigations, e.g. \cite{horn2014degree, jiang2013degree, zhu1998chromatic, rodl2008ramsey,conlon2023three, nevsetvril1976ramsey, folkman1970graphs, fox2016minimum, han,bamberg2022minimum}. 
	%TSz: THIS IS CERTAINLY NOT TRUE FOR every H: For example, it follows from the work of R\"odl and Siggers~\cite{RS08} that $\cM_r(H)$ is infinite. TSz: I would delete as a somewhat random example.
	
	%We are particularly interested in the case when the graph parameter is the minimum degree and $H = K_{k+1}$, the complete graph on $k+1$ vertices.  
	In this paper we will be particularly interested in the minimum degree of minimal Ramsey graphs. For a graph $H$ and number $r$ of colors we define %use the following definition for any two positive integers $k$ and $r$
	\[s_r(H) := \min\{\delta(G) \,\, | \,\, G \in \cM_r(H)\},\]
	to be the smallest possible minimum degree that could occur among minimal $r$-Ramsey graphs for $H$.  For cliques in the classical two-color case, Burr et al.\ ~\cite{burr1976graphs} established the following precise result: 
	\begin{equation}
		s_2(K_{k})=(k-1)^2.
	\end{equation}
	Upon first glance, this result looks extremely surprising. 
	%for a couple of reasons.
	First, it determines the {\em exact} value of the smallest possible minimum degree in a minimal 2-Ramsey graph for $K_{k}$. This is in stark contrast with our knowledge about the smallest {\em number of vertices} in such graphs, which is hopelessly out of reach. Furthermore, the value of the smallest minimum degree turns out to be just a quadratic function of $k$. This is incredibly small considering that we know that even the smallest of Ramsey graphs will have exponentially many vertices. How could it then be possible to create a (necessarily enormous) $2$-Ramsey graph for $K_{k}$, that has a vertex with just $(k-1)^2$ neighbors, such that the presence of this vertex, and in fact any of its incident edges are {\em crucial} in guaranteeing the $2$-Ramsey-ness of said enormous graph?
	
	Fox, Grinshpun, Liebenau, Person and Szab\'o~\cite{fox2016minimum} investigated the behaviour $s_r(K_{k})$ for more than two colors. They found that for any fixed clique order $k\geq 3$ there exist positive constants $c_k,C_k$, such that 
	\begin{equation} \label{eq:fox-constantk}
		c_kr^2\frac{\log r}{\log\log r} \leq s_r(K_{k}) \leq C_kr^2(\ln r)^{8k^2}.
	\end{equation}
	For the triangle $K_3$, a slightly stronger lower bound was given in \cite{fox2016minimum}, which was proved to be tight up to a constant factor by Guo and Warnke~\cite{guo2020packing}. 
	\begin{equation}
		s_r(K_{3}) = \Theta (r^2 \log r).
	\end{equation}
	These results establish that for any fixed clique order $k$ the value of the smallest minimum degree $s_r(K_k)$ is quadratic in the number $r$ of colors, up to some logarithmic factor. The power of the logarithm in the gap between the upper and lower bounds however depends significantly on $k$. 
	
	On the other end of the spectrum, when the number $r$ of colors is constant, H\`an, R\"odl, and Szab\'o ~\cite{han} determined the order of magnitude of the smallest minimum degree $s_r(k)$, up to a logarithmic factor. 
	More generally, they have shown that there exists a constant $C$ such that for every $k^2 >r$
	%there exists a constant $C_r$ such that 
	%\begin{equation}\label{eq:han-constantk}
	%  s_r(K_{k}) \leq C_r k^2 \log^2 k.  
	%\end{equation}
	% Since one can see that $s_2(K_k) \leq s_r(K_k)$, coupled with the result of Burr et al.\ above, this result establishes that having more than two colors does not increase the smallest minimum degree $k^2$ by much, by at most a $C_r\log^2 k$-factor only.  
	\begin{equation} \label{eq:han-general}
		s_r(K_k) \leq C r^3k^2 \log^3 r \log^2 k.
	\end{equation}
	Considering that we know from \cite{fox2016minimum} and \cite{burr1976graphs} that
	$s_r(K_k) \geq s_2(K_k) = (k-1)^2$, the bound (\ref{eq:han-general}) establishes that $s_r(K_k)$ is quadratic up to a $\log^2$-factor for any fixed number $r$ of colors.
	
	When both $k$ and $r$ are increasing, say $r=r(k)$ is a decent increasing function of $k$, the best known upper bounds vary depending on how fast $r$ grows. In the range $r(k) < k^2$, the upper bound of (\ref{eq:han-general}) is the best we know.
	
	In the complementary range of $r(k)\geq k^2$ another construction of Fox et al.\ \cite{fox2016minimum} gives\footnote{In \cite{fox2016minimum} only the weaker upper bound $s_r(K_k) \leq q^3=O(r^3k^6)$ is stated. However, Bishnoi and Lesgourgues~\cite{bishnoi2022new} recently observed that the choice $q\sim rk^2$ used in \cite{fox2016minimum} for the parameter $q$ is suboptimal and the calculation there also works with $q \sim rk \log k$, resulting in (\ref{eq:fox-polynomial}).} 
	\begin{equation}\label{eq:fox-polynomial}
		s_r(K_k) = O(r^3k^3 \log^3 k).
	\end{equation}
	Bamberg, Bishnoi, and Lesgourgues~\cite{bamberg2022minimum} developed a generalization of this construction and used that to obtain
	\begin{equation}\label{eq:bamberg}
		s_r(K_k) = O(r^{5/2}k^5).
	\end{equation}
	This represents the best known upper bound when $r(k)\geq \Omega \left(\frac{k^4}{\log^6 k}\right)$ and $k$ tends to infinity.\\
	
	The minimum degree of minimal Ramsey graphs has also been the subject of considerable study beyond the quantitative behavior of $s_r(K_k)$: \cite{bishnoi2023minimum, clemens2015minimum} deal with the generalizations of the parameter to asymmetric settings and hypergraphs, \cite{szabo2010minimum, boyadzhiyska2025ramsey, grinshpun2017minimum} tackle the question of \textit{$q$-Ramsey simplicity}, while \cite{boyadzhiyska2022minimal} investigates the {\em number} of vertices that can attain the degree $s_r(H)$.

	\subsection{Our results} 
	
	\subsubsection{Upper bounds}
	Summarizing the above: (1) When either the order $k$ of the clique or the number $r$ of colors is constant, the smallest minimum degree $s_r(K_{k})$ is quadratic, up to poly-logarithmic factors, in terms of $r$ and $k$, respectively; (2) when $k$ and $r$ both tend to infinity, the best known upper bounds are polynomial, with the degree of $r$ (and sometimes also of $k$) being more than two.  
	Bamberg et al.\ \cite{bamberg2022minimum} in fact conjectured that an upper bound $r^2k^2$, up to logarithmic factors, should hold in all ranges of the parameters. 
	
	In this paper we give new constructions which establish this for a large range of the parameters and improve the best known upper bounds for every large enough $k$ and $r\geq c \frac{k}{\log^2 k}$.
	
	Our first main theorem improves the bounds (\ref{eq:bamberg}), (\ref{eq:fox-polynomial}) and (\ref{eq:han-general}) whenever $k$ tends to infinity and $r$ is large enough.
	\begin{theorem}\label{thm: MinDegGen}
		For all sufficiently large $k,r$ satisfying $k \leq r \log^2 r$, we have 
		\[s_r(K_{k}) \leq 2^{400} k^2r^{2+\frac{30}{k}} \log^{20}r\log^{20}k.\]
	\end{theorem}
	Note that this upper bound is of the form $(rk)^{2+o(1)}$ and the error term becomes logarithmic when $r(k) = e^{O(k\log k)}$. 
	
	For constant $k$, our second main result reduces the power of the log factor in the upper bound of (\ref{eq:fox-constantk}) from $8k^2$ to $2$.
	\begin{theorem}\label{thm: MinDegConst}
		For all $k \geq 3$ there exists a constant $C_k$ such that for all $r \geq 2$
		\[s_r(K_{k}) \leq C_k (r\log r)^2. \]
	\end{theorem}
	Combined with the lower bound of (\ref{eq:fox-constantk}), Theorem~\ref{thm: MinDegConst} determines the value of $s_r(K_k)$ up to a factor $O(\log r \log\log r)$, for every fixed $k\geq 4$. 
	
	\subsubsection{Colored semisaturation numbers}
	
	Tran~\cite{tran2022two} observed that a certain graph parameter, introduced by Dam\'asdi et al.\ ~\cite{damasdi2021saturation}, related to colored saturation, is also relevant for $s_r(K_k)$.
Given integers $r,k \geq 2$, let ${\cal RC}_r(K_k)$ be the set of edge $r$-colorings of complete graphs, for which any extension to an edge $r$-coloring of a complete graph of one larger order creates a new monochromatic copy of $K_k$. The $r$-color semisaturation number $\mathrm{ssat}_r(K_{k})$ is defined to be the {\em smallest} order $n$ such that there exists an edge $r$-coloring of $E(K_n)$ in ${\cal RC}_r(K_k)$. 
	This quantity was first investigated, within a more general framework, by Dam\'asdi et al.\ \cite{damasdi2021saturation}. Tran \cite{tran2022two} observed that 
	\begin{equation} \label{eq:semisat}
		\mathrm{ssat}_r(K_k) \leq s_r(K_k)
	\end{equation}
	and asked \cite[Question 4.2]{tran2022two} whether there exists a constant $C$ (independent of $k$) such that $\mathrm{ssat}_r(K_{k}) = O_k(r^2(\log r)^C)$.
	%motivated by earlier upper bounds on $s_r(K_k)$. 
	Our \Cref{thm: MinDegConst} together with (\ref{eq:semisat}) answers this question in the affirmative. 
	\begin{corollary}
		For all $k \geq 4$, there exists a constant $C_k$ such that  $\mathrm{ssat}_r(K_k) \leq C_kr^2 \log^2 r$.
	\end{corollary}
	
	It turns out however that for $\mathrm{ssat}_r(K_k)$ one can prove an even stronger bound. 
	In the same paper, Tran asks whether $\mathrm{ssat}_r(K_{k+1}) = \omega(r^2)$ as $r \to \infty$ \cite[Question 4.1]{tran2022two}. The following theorem answers this question in the negative.
	\begin{theorem}\label{thm:semisaturation}
		For all $k,r \geq 2$, we have $\mathrm{ssat}_r(K_{k}) \leq 4(k-2)^2r^2$.
	\end{theorem}
	When $k$ is fixed and $r$ goes to infinity, \Cref{thm:semisaturation} together with the lower bound of \cite{fox2016minimum} from (\ref{eq:fox-constantk}) establishes a separation by a factor $\frac{\log r}{\log\log r}$ between the orders of magnitude of $\mathrm{ssat}_r(K_{k})$ and $s_r(K_k)$.
	%  Recall that Fox et al \cite{fox2016minimum} proved that $s_r(K_k)\geq c_k r^2\ln r/\ln \ln r$ for some $c_k >0$. This alongside \Cref{thm:semisaturation} already show an asymptotic separation between $\mathrm{ssat}_r(K_{k})$ and $s_r(K_k)$ when $k$ is fixed and $r$ grows arbitrarily.
	%The main objective of this paper is to study the more general multi-color setting of the same parameter. Instead of dealing with $s_r(K_k)$ directly however, we work with an equivalent combinatorial parameter first mentioned in ~\cite{burr1976graphs}, later formalized explicitly in ~\cite{fox2016minimum}, and for which we need a couple of definitions.   

	\subsubsection{Lower bounds}
	Except for the (nearly) extremal ranges, i.e. when either the clique order $k$ or the number $r$ of colors is (nearly) constant, we do not have a lower bound on $s_r(k)$ which is quadratic both in $k$ and $r$ up to poly-logarithmic factors. 
	Damásdi et al.\ ~\cite{damasdi2021saturation} gave a lower bound of %{\bf Please check this:} $(r-1)k^2 - 3rk +4k +2r-3 =  
    $\Omega (k^2r)$ on $\mathrm{ssat}_r(K_k)$ that transfers to a lower bound on $s_r(K_k)$ via the observation (\ref{eq:semisat}) of Tran~\cite{tran2022two}.
	%{\bf Is your bound better?}
    %\red{To use that paper we use $c = r$ and $k_1 = \dots = k_c = k$ and hence I find $(k-1)((r-1)k-2r+3)$ which is not quite the same right?}
	
	Here we prove a lower bound that is quadratic in $r$ and linear in $k$. 
	
	\begin{theorem}\label{thm: LowerKk+1}
		For all $k,r \geq 3$, we have $s_r(K_{k}) = \Omega(kr^2)$.
	\end{theorem}

	\section{On the proof}
	Instead of dealing with minimal Ramsey graphs, the proofs of all the known bounds on $s_r(K_k)$ work with an alternative function, distilled by Fox et al.\ \cite{fox2016minimum} from the original argument of Burr et al.\ \cite{burr1976graphs} for $s_2(K_{k}) =(k-1)^2$. 
	
	\begin{definition}[Color Pattern.]
		A sequence of pairwise edge-disjoint graphs $G_1,\dots,G_r$ on the same vertex set $V$ is called an \textit{$r$-color pattern} on $V$ (where the edges of $G_i$ are said to have color $i$). The color pattern is \textit{$K_{k+1}$-free} if $G_i$ is $K_{k+1}$-free for every $i=1, \ldots r$.\\  
		Given a color pattern $G_1,\dots,G_r$ on the vertex set $V$ and an $r$-coloring $c:V\rightarrow [r]$ of the vertices, a \emph{strongly monochromatic} copy of a graph $H$ according to $c$ is a copy of $H$ whose edges and vertices all have the same color.
	\end{definition}
	
	\begin{definition}
		Let $r,k \geq 2$ be positive integers, we define $P_r(k)$ to be the smallest positive integer $n$ such that there exists a $K_{k+1}$\textit{-free} color pattern $G_1,\dots,G_r$ on the vertex set $[n]$ such that every $r$\textit{-coloring} of $[n]$ induces a strongly monochromatic $K_k$.
	\end{definition}
	The connection between $s_r$ and $P_r$ is summarized in the following lemma.
	\begin{lemma}\cite[Theorem 1.5]{fox2016minimum} \label{lem: ColorPattern}
		For all integers $r,k \geq 2$ we have
		$s_r(K_{k+1})=P_r(k).$
	\end{lemma}
	
	%Tibor: It is NOT because we want to ''ease." It is because we have NO OTHER IDEA. 
	%To ease the study of $P_r(k)$, 
	
	To prove an upper bound on $P_r(k)$, one needs to construct a $K_{k+1}$-free $r$-color pattern $G_1, \ldots , G_r$ with the specific property about strongly monochromatic $K_k$. 
	As it happens, at the moment we have no other idea of guaranteeing the existence of a strongly monochromatic $K_k$ in an arbitrary $r$-coloring of the vertices, but requiring that {\em each} of the graphs $G_i$ has a $K_k$ in {\em each} subset of size at least $n/r$ and then use this for the largest color class in $[n]$. 
	To this end, we define for a graph $G$ and a positive integer $k$ the parameter $\alpha_k(G)$ to be the order of the largest $K_k$\textit{-free} induced subgraph of $G$. 
	
	\begin{observation}\cite[Lemma 4.1]{fox2016minimum} \label{obs: ColorPatter}
		If there exists a $K_{k+1}$-free color pattern $G_1, \ldots , G_r$ on $[n]$ such that $\alpha_k(G_i) <\frac{n}{r}$ for every $i=1, \ldots , r$, then 
		$s_r(K_{k+1})=P_k(r) \leq n$.
	\end{observation}
	
	We will also follow this road and construct $K_{k+1}$-free $r$-color patterns on $[n]$ with $\alpha_k$-values less than $n/r$. 
	Before constructing $r$ edge-disjoint $K_{k+1}$-free graphs with $\alpha_k < n/r$ however, one better deals with the ''simpler'' problem of constructing just one. 
	This is exactly the task of the well-studied Erd\H os-Rogers function $f_{k,k+1}(n)$ which asks for the smallest value of $\alpha_k(G)$ of $K_{k+1}$-free graphs $G$ on $n$ vertices. 
	Given a good Erd\H os-Rogers graph, one then ''only'' has to pack as many of them as possible on $n$ vertices.  
	Indeed, Fox et al.\ ~\cite[Conjecture 5.2]{fox2016minimum} even predicted that for every fixed $k\geq 3$ we will have $P_r(k) = \Theta (r \cdot (f_{k,k+1}(r))^2)$. Considering the recent improvements of Mubayi and Verstraete~\cite{mubayi2024order} on the Erd\H os-Rogers function, Theorem~\ref{thm: MinDegConst} comes within a log-factor of resolving this conjecture.
%will also follow this road and construct $K_{k+1}$-free $r$-color patterns on $[n]$ with $\alpha_k$-values less than $n/r$. 
%	A ''simpler'' problem closely related to this is the following: For which $n$ can we construct only a single graph $G_1$ with $\alpha_k(G_1) < n/r$? This is exactly the problem of the well-studied Erd\H os-Rogers function $f_{k,k+1}(n)$ which asks for the smallest value of $\alpha_k(G)$ of $K_k$-free graphs on $[n]$. 
%	Given a good Erd\H os-Rogers graph, one ''only'' has to tile a large fraction of the complete graph with it. Indeed, Fox et al.\ ~\cite[Conjecture 5.2]{fox2016minimum} even predicted that for every fixed $k\geq 3$ we will have $s_r(K_k) = \Theta (r \cdot (f_{k-1,k}(r))^2)$. Considering the recent improvements of Mubayi and Verstraete~\cite{mubayi2024order} on the Erd\H os-Rogers function, Theorem~\ref{thm: MinDegConst} comes within a log-factor of resolving this. 
	Good constructions for the Erd\H os-Rogers function will be extremely useful for us as well, but creating color patterns using them requires additional ideas. 
    
%	The generic Erd\H os-Rodgers construction starts with an appropriate linear hypergraph on $n$ vertices and creates a graph by dropping a random $K_{k+1}$-free graph on each of its hyperedges. 
%        This random choice has to balance that no $K_{k+1}$ is created from the edges coming from within different hyperedges, but there are enough edges so that any $n/r$-subset of the vertices contains a $K_k$. 
	
	The general approach to construct the desired color patterns is to start with an ''appropriate'' $u$-uniform linear hypergraph on $[n]$ with essentially as many hyperedges as possible, that is, in the order $\frac{n^2}{u^2}$. Then one assigns one of $r$ colors to each hyperedge ''appropriately'', to indicate which color pairs of vertices inside the hyperedge will receive should they be chosen to be an edge at all. Here the linearity of the hypergraph plays a crucial role: every pair of vertices belongs to (at most) one hyperedge. Finally, one constructs the graphs $G_i$ by dropping an ''appropriately'' random $k$-partite graph within each hyperedge of assigned color $i$, where the choices for different hyperedges are usually independent. 
    This random choice must balance that no $K_{k+1}$ is created from the edges coming from within different hyperedges, yet there are enough edges so that any $n/r$-subset of the vertices contains a $K_k$. The crux of the matter is how to define the various occurrences of "appropriate" above so that they complement each other well.
	
	In the construction of Fox et al.\ ~\cite{fox2016minimum} for (\ref{eq:fox-constantk}) (crucially making use of the Erd\H os-Rogers construction of Dudek, Retter, and R\"odl~\cite{dudek2014generalized})
	and that of H\`an et al.\ ~\cite{han} for (\ref{eq:han-general}) the linear hypergraph is essentially given by the lines of an (affine or projective) plane of order $q$ and the ''appropriate'' color assignment chosen randomly. In the constructions of Fox et al.\ ~\cite{fox2016minimum} for (\ref{eq:fox-polynomial}) and of Bamberg et al.\ ~\cite{bamberg2022minimum} for (\ref{eq:bamberg}) the linear hypergraph is given by some (pseudo)lines in a higher dimensional space and the color-assignment is defined algebraically. The point of these assignments is to ensure that the hypergraph of each color class is {\em triangle-free}, hence the $K_{k+1}$-freeness of each $G_i$ will be automatic once the graphs inside the hyperedges are $K_{k+1}$-free.
	
	In our construction we also start with the projective plane, working in the dual setup, so the lines will correspond to vertices and the vertices correspond to the hyperedges. We choose the order to be $q^2$, so we are able to make use of Hermitian unitals and its beneficial algebraic and combinatorial properties. One of the main ideas of our construction is to combine the probabilistic color assignment to the hyperedges with an appropriate algebraic one. Unlike in \cite{fox2016minimum} and \cite{bamberg2022minimum}, our algebraic color assignment will not guarantee immediate $K_{k+1}$-freeness, but will however ensure that the analysis of $K_{k+1}$-freeness will only have to consider very limited types of forbidden events. The random part of the color assignment then helps to limit the number of bad events within those types. 

    Hermitian unitals, which dictate the rigid structures of our monochromatic $K_{k+1}$'s, have been instrumental in recent results in Ramsey theory due to the second and fourth authors \cite{mattheus2024asymptotics}, Erd\H{o}s-Rogers functions due to Mubayi and the fourth author \cite{mubayi2024order}, recent results on the generalization of this function studied by Balogh, Chen and Luo \cite{balogh2025maximum} and Mubayi and the fourth author \cite{mubayi2024erd}, and improvements for the Erd\H{o}s-Rogers function $f_{k,k+2}(n)$ by Janzer and Sudakov \cite{janzer2025improved}. \\
	
	%The main novelty compared to the aforementioned applications is that we have to use multiple disjoint Hermitian unitals to reach the desired number of colors. 
    The organization of the paper is the following. In \Cref{sec:pencilunitals} we discuss the algebraic content of our color assignment, which is based on the use of multiple disjoint Hermitian unitals. 
	The probabilistic refining of the algebraic coloring is the subject of \Cref{Coloring II: Probability}. In \Cref{sec:graphs}, we complete the proofs of Theorem~\ref{thm: MinDegGen} and \ref{thm: MinDegConst} by arguing that our graphs are indeed likely to be $K_{k+1}$-free while maintaining a small $\alpha_k$. The proof of our lower bound in Theorem~\ref{thm: LowerKk+1} is given in \Cref{sec:lowerbounds}, while Section~\ref{section: Tran} contains the short proof of Theorem~\ref{thm:semisaturation}. Section~\ref{sec:remarks} collects a number of probelms remaining open.

	\section{Coloring I: Finite Geometry}\label{sec:pencilunitals}
	We start by describing a classical technique to partition the points of the projective plane $\PG(2,q^2)$ based on a pencil of $q$ Hermitian unitals sharing a common tangent line. Recall that a Hermitian unital $U$ is defined by a non-singular Hermitian matrix $A$ over $\mathbb{F}_{q^2}$, i.e. $A^q = A^\top$, so that $H:x^\top A x^q = 0$. For example, the matrix 
	\[A = \begin{pmatrix} 1 & 0 & 0 \\ 0 & 0 & 1 \\ 0 & 1 & 0 \end{pmatrix}\]
	defines the Hermitian unital with equation $X^{q+1}+YZ^q+Y^qZ=0$. It is well-known, see for example \cite{BarwickEbert}, that every line intersects a Hermitian unital in either $1$ or $q+1$ points. We will say that a line is \textit{tangent} or \textit{secant} respectively. Some more combinatorial properties we need are the following. We refer to the book by Barwick and Ebert \cite{BarwickEbert} for these and many more properties of Hermitian unitals.
	
	\begin{lemma}\label{lem:unitalcombinatorics}
		Let $U$ be a Hermitian unital in $\PG(2,q^2)$, then we have the following properties:
		\begin{enumerate}
			\item $|U| = q^3+1$,
			\item there are $q^4-q^3+q^2$ secant and $q^3+1$ tangent lines to $U$,
			\item each point on the unital is incident to a unique tangent line and $q^2$ secant lines, and 
			\item \emph{~\cite{mubayi2024order}} for every $k\geq 3$ secants pairwise intersecting in $U$, there exists a point in $U$ incident to at least $k-1$ of them.
		\end{enumerate}
	\end{lemma}
	
	The idea is that given any Hermitian unital $U$ and a tangent line $\ell_\infty$ at the point $p_\infty$, we can consider the defining equations of both and define (with some abuse of notation) the pencil 
	\[\cU = \{U_\lambda:U+\lambda\cdot(\ell_\infty)^{q+1} = 0 \,\, | \,\, \lambda \in \Fq\}\] 
	We will see that each $\lambda \in \Fq$ defines a unital $U_\lambda$, where $U_0 = U$. Moreover, an elementary calculation will show that every line in $\PG(2,q^2)$ not through $p_\infty$ is tangent to exactly one unital in $\cU$ and secant to the $q-1$ others.
	
	These properties can be derived purely geometrically, but for the sake of concreteness, we will use coordinates. So consider $U:X^{q+1}+YZ^q+Y^qZ=0$, then it is easy to check that $Z = 0$ is a tangent line at the point $(0,1,0)$. We will denote them as $\ell_\infty$ and $p_\infty$.
	
	\begin{lemma}\label{lem:unitalpencil}
		Consider the set $\cU = \{U_\lambda\}_{\lambda \in \Fq}$ of unitals in $\PG(2,q^2)$ defined by 
		\[U_\lambda:X^{q+1}+YZ^q+Y^qZ+\lambda Z^{q+1} = 0.\]
		Then 
		\begin{enumerate}
			\item $\bigcup_{\lambda \in \Fq}\left(U_\lambda \setminus p_\infty\right) \cup \ell_\infty$ is a partition of the points of $\PG(2,q^2)$.
			\item Every line in $\PG(2,q^2)$ not through $p_\infty$ is tangent to exactly one unital and secant to all others.
		\end{enumerate}
	\end{lemma}
	
	\begin{proof}
		Every $U_\lambda$ is defined by the non-singular Hermitian matrix
		\[\begin{pmatrix}
			1 & 0 & 0 \\ 0 & 0 & 1 \\ 0 & 1 & \lambda
		\end{pmatrix}\]
		and is hence a Hermitian unital.
		
		Observe that $p_\infty \in U_\lambda$ for all $\lambda \in \Fq$ and this is the only common point of any two unitals in the pencil. Now given a point not on $\ell_\infty$ with homogeneous coordinates $(x,y,z)$, so that $z \neq 0$, it is clear that both $a:=x^{q+1}+yz^{q}+y^qz$ and $b:=z^{q+1} \neq 0$ are elements of $\Fq$, so that there is exactly one solution in $\Fq$ to the equation $a+\lambda b = 0$. This implies that every point not on $\ell_\infty$ is contained in exactly one unital of the pencil.
		
		Finally, any line $\ell$ not through $p_\infty$ intersects every $U_\lambda$ in either $1$ or $q+1$ as each of them is a Hermitian unital. So let $t$ and $s$ be the number of times that each case occurs. Since there are $q^2$ points on $\ell$ not on $\ell_\infty$, we see
		\begin{align}
			\begin{cases}
				t+s=q \\
				t+s(q+1) = q^2,
			\end{cases}
		\end{align}
		and hence $t = 1$, $s = q-1$.
	\end{proof}
	
	\begin{corollary}\label{lem:countcommonsecants}
		Let $\Lambda \subset \Fq$, then there are $q^4-|\Lambda|q^3+q^2$ common secants to the set of unitals $\{U_\lambda\}_{\lambda \in \Lambda}$.
	\end{corollary}
	\begin{proof}
		Denote by $\cT_\lambda$ the set of tangent lines to $U_\lambda$. By \Cref{lem:unitalpencil} we see that for distinct $\lambda,\lambda' \in \Lambda$ we have $\cT_\lambda \cap \cT_{\lambda'} = \{\ell_\infty\}$. We know that $|\cT_\lambda| = q^3+1$ by \Cref{lem:unitalcombinatorics} and hence $|\cup_{\lambda \in \Lambda}\cT_\lambda| = |\Lambda|q^3+1$ so that there are $(q^4+q^2+1)-(|\Lambda|q^3+1)$ lines which are secant to every $U_\lambda$, $\lambda \in \Lambda$.
	\end{proof}
	We hereby fix for every prime power $q$ an arbitrary subset $\Lambda:= \Lambda(q) \subset \mathbb{F}_q$ of size $\lfloor \frac{q}{2}\rfloor$.  We denote by $P=P(q):= \underset{\lambda \in \Lambda}{\cup} U_\lambda -\{p_\infty\}$ the union of the unitals indexed by $\Lambda$ except for $p_\infty$, by $L=L(q)$ the set of their common secants, and by $\mathcal{P}=\mathcal{P}(q):= \{U_\lambda\}_{\lambda \in \Lambda}$ the $\Lambda$\textit{-restricted} pencil. Lemma \ref{lem:countcommonsecants} and \Cref{lem:countcommonsecants} assert that
	
	\begin{equation*}
		q^4/2-q^3\leq |\Lambda|q^3 = |P| \leq q^4/2 \quad \text {     and       } \quad q^4/2 \leq |L| \leq q^4.
	\end{equation*}
	Later on we will construct  $K_{k+1}$\textit{-free} graphs using the disjoint unitals $\{U_\lambda\}_\lambda$ of $\mathcal{P}$. Before we do so, we need to partition each single unital, only this time the partition is done probabilistically.

	\section{Coloring II: Probability} \label{Coloring II: Probability}
	The starting point here is the $\Lambda$\textit{-restricted} pencil $\mathcal{P}$. Inside each $U_\lambda$, we color its points uniformly at random with $c$ colors and repeat this for all $U_\lambda$, $\lambda \in \Lambda$, using a different set of $c$ colors for each unital. In this way we obtain a coloring of $P$ with $|\Lambda|c$ colors, which is well-defined since the unitals $\{U_\lambda\}_{\lambda \in \Lambda}$ are disjoint, except for the point $p_{\infty}$ which we exclude permanently. For each color $i$, let $P_i$ be the set of points with color $i$. We first establish a simple lemma for which we need the Chernoff bound.
	
	\begin{proposition}[Chernoff bound]
		Let $Z$ be a binomial random variable with mean $\mu$. Then for any real $\varepsilon \in [0,1]$,
		\begin{align*}
			&\mathrm{Pr}(Z > (1+\varepsilon)\mu) \leq \exp\left(\frac{-\varepsilon^2\mu}{4}\right) \,\,\,\text{ and } \\
			&\mathrm{Pr}(Z < (1-\varepsilon)\mu) \leq \exp\left(\frac{-\varepsilon^2\mu}{2}\right).
		\end{align*}
	\end{proposition} 
	
	\begin{lemma}\label{lem:coloringpoints}
		For any integer $c \leq q/(48\log q)$, there exists a $|\Lambda|c$\textit{-coloring} of $P$ such that for all colors $i$ the following holds.
		\begin{center}
			\begin{tabular}{lp{5in}}
				$(1)$ & The set $P_i$ has size at most $2 q^3/c$ and at least $q^3/2c$. \\
				$(2)$ & Every line in $L$ contains at least $q/2c$ and at most $2q/c$ points of $P_i$. \\
			\end{tabular}
		\end{center}
	\end{lemma}
	
	\begin{proof}
		Suppose that $P_i \subseteq U_\lambda$ for some $\lambda \in \Lambda$. The probability that a given point of $U_\lambda \setminus \{p_\infty\}$ receives color $i$ is exactly $1/c$, independently of the other points. Let $A_i$ be the event that $|P_i| \geq 2q^3/c$ or $|P_i| \leq q^3/2c$. The Chernoff bound then shows with $\varepsilon = 1$ and $\varepsilon = 1/2$ respectively that
		\[\mathrm{Pr}(A_i) = \mathrm{Pr}(|P_i| \geq 2q^3/c) + \mathrm{Pr}(|P_i| \leq q^3/2c) \leq \exp\left(-\frac{q^3}{4c}\right) + \exp\left(-\frac{q^3}{8c}\right) \leq 2\exp\left(-\frac{q^3}{4c}\right).\]
		It follows by the union bound that the probability of the event $\cup_i A_i$ is at most
		\[c|\Lambda|\left(2\exp\left(-\frac{q^3}{4c}\right)\right) \leq \frac{q^2}{48\log q}\exp\left(-\frac{12q^2}{\log q}\right) < \frac{1}{2}, \]
		and hence with probability more than $1/2$, none of the events $A_i$ occur.
		
		Similarly, given a line $\ell \in L$ and a color $i$, we find that
		\begin{align}
			\mathrm{Pr}(|\ell \cap P_i| < q/2c) &\leq \exp\left(-\frac{q}{8c}\right) \\
			\mathrm{Pr}(|\ell \cap P_i| > 2q/c) &\leq \exp\left(-\frac{(q-1)^2(q+1)}{4c(q+1)^2}  \right) \leq \exp\left(-\frac{q}{8c} \right).
		\end{align}
		%\red{where does the expression in the second line come from, is it not just $\mu = (q+1)/c$? as $P_i \subset U_\lambda$ for some $\lambda \in \Lambda$ and hence $\ell$ intersects it in $q+1$ points.}
        %\textcolor{blue}{YA: It was ugly and inaccurate, now I changed it so that it's just ugly. I write the bound in terms of $q/c$ simply because I dont want to write $(q+1)/c$ for the remainder of the paper.}
        %\red{The latter I definitely agree with.}

		We now use the union bound to conclude that the probability there exists a line in $L$ whose intersection with some color class is abnormal is at most
		\[2|L|\cdot c|\Lambda|\cdot \exp\left(-\frac{q}{8c}\right) \leq \frac{q^6}{48\log q}\exp(-6\log q) < \frac{1}{2},\]
		and therefore, with non-zero probability, there is a coloring which satisfies the required properties.
	\end{proof}
	
	For every $c \leq q/(48\log q)$ we fix a coloring as guaranteed by the preceding lemma. We then define for each color $i$ of the $|\Lambda|c$ colors the graph $\tilde{G}_i$ to be the graph whose vertex set is $L$ and whose edge set consists of pairs $\{\ell,\ell'\}$ such that $\varnothing \neq\ell \cap \ell' \subset P_i$. The graphs $\{\tilde{G}_i\}_i$ certify the following lemma:
	\begin{lemma} \label{lem: graphlemma}
		Let $q$ be a prime power, $c\leq \frac{q}{48\log q}$ be an integer, and $m :=c\lfloor\frac{q}{2}\rfloor$. There exists an integer $q^4/2\leq n \leq q^4$ and edge-disjoint graphs $\tilde{G}_1,\dots,\tilde{G}_m$ on the vertex set $[n]=L$ where the following  is true for every $i \in [m]$ and $k\geq 3$:
		\begin{enumerate}
			\item $\tilde{G}_i$ is the union of of edge-disjoint maximal cliques called \em{point-cliques}.
			\item The number of point-cliques is at least  $q^3/2c$ and at most $2q^3/c$.
			\item Every vertex $\ell \in L$ is a member in at least $q/2c$ and at most $2q/c$ point-cliques.
			\item For every $K_{k+1}$ in $\tilde{G}_i$, there exists a point-clique containing either exactly $k$ or exactly $k+1$ points of the clique. We call cliques of the former type $(k+1)$\textit{-fans} and from the latter degenerate.
			\item Every edge $e \in E(\tilde{G}_i)$ is contained in at most $2(2q/c)^k$ many $(k+1)$\textit{-fans}.
		\end{enumerate}
	\end{lemma}
	\begin{proof} The first four points are immediate from the previous discussion.  For the last point, consider an edge $e \in E(\tilde{G}_i).$ The edge $e$ corresponds to two secants $\ell_1,\ell_2\in L$ such that $\varnothing\neq\ell_1\cap\ell_2\subset P_i.$   There are two types of $(k+1)$\textit{-fans} in $\tilde{G}_i$ containing both $\ell_1,\ell_2$: Fans with $\ell_1\cap\ell_2$ as their concurrence point, of which there are at most $(|\ell_1|-1)(|\ell_2|-1)\binom{2q/c}{k-2}$; and fans with other concurrence point, of which there are at most $(|\ell_1|+|\ell_2|-2)\binom{2q/c}{k-1}$.
	\end{proof}
	\begin{remark}\label{Rem:MVConstant}
		Mubayi and the fourth author \cite{mubayi2024order} proved that for every $a\geq 128$ and sufficiently large $q$ so that $q \geq a \log q$, any graph on $n \in [q^4/2,q^4]$ vertices satisfying $(1)-(5)$ with $c=\lceil\frac{q}{a \log q}\rceil$ contains a $K_{k+1}$\textit{-free} spanning subgraph $H$ such that $\alpha_k(H)\leq 2^{40k+1}aq^2\log q$.
	\end{remark}
	
	\section{Proof of Main Theorems}\label{sec:graphs}
	We are now ready to prove \Cref{thm: MinDegGen,thm: MinDegConst}. The proof of \Cref{thm: MinDegConst} is now straightforward and should serve as a warm up for the more involved proof of \Cref{thm: MinDegGen}.
	
	%comment the old formulation of the theorem out.
	
	\iffalse
	\begin{reptheorem}{thm:main}
		Let $k \geq 3$ and $q$ a prime power such that $q \geq 2^{10}k\log q$. For all non-negative integers $s \leq q^2/(2^{12}k\log q)$ there exist edge-disjoint $K_{k+1}$-free graphs $G_1,\dots,G_s$ on the same vertex set $[N]$, $q^4/2 \leq N \leq q^4$, with the additional property that every subset of size at least $2^{100k}q^2\log q$ induces a copy of $K_k$ in each $G_i$.
	\end{reptheorem}
	\fi

	\begin{proof}[Proof of \Cref{thm: MinDegConst}]
		Fix some $k \geq 3$, set $C_k=2^{300k}$, and let $r$ be sufficiently large (so that $\sqrt{r}\geq 128\log r$ is satisfied.) By \Cref{obs: ColorPatter}, it suffices to find a $K_{k+1}$\textit{-free} color pattern $G_1,\dots,G_r$ on a vertex set of size $n\leq C_kr^2\log^2r$ such that $\alpha_k(G_i)<n/r$ for all $i \in [r]$. By Chebyshev's theorem, there exists a prime number $q$ satisfying \[\frac{1}{2}\Bigl(C_kr^2\log^2r\Bigr)^{1/4}\leq q \leq \Bigl( C_k r^2 \log^2r \Bigr)^{1/4}.\] Set $c=\lceil \frac{q}{128\log q}\rceil$ and note that 
		\[c\Bigl\lfloor\frac{q}{2}\Bigr \rfloor \geq \frac{q^2}{2^9 \log q} \geq \frac{2^4q^2}{\sqrt{C_k} \log r} \geq r.\]
		Therefore applying \Cref{lem: graphlemma} with the above value of $c$ gives us at least $r$ edge disjoint graphs $\tilde{G}_1,\dots,\tilde{G}_r$ on a shared vertex set of size $\frac{q^4}{2}\leq n\leq q^4=C_k(r \log r)^2$ satisfying $(1)-(5)$ of \Cref{lem: graphlemma} above. By \Cref{Rem:MVConstant}, we can find for each $i \in [r]$ a $K_{k+1}$\textit{-free} subgraph $G_i \subset \tilde{G}_i$ such that
		\[\alpha_{k}(G_i) \leq 2^{40k+8}q^2 \log q \leq 2^{40k+8}\sqrt{C_k} \:r\log r \Bigl(\log r +\log C_k \Bigr)  < \frac{n}{r}. \]
		The graphs $G_1,\dots,G_r$ certify the assertion of \Cref{thm: MinDegConst}.
	\end{proof}

	Note that the use of \Cref{Rem:MVConstant} necessitates the exponential dependency on $k$. To circumvent this by-product we alter our construction slightly; in particular, we modify the random sparsification used in \Cite{mubayi2024order} to prove \Cref{thm: MinDegGen}. The starting point of the proof is again \Cref{lem: graphlemma} and in particular item (4): Every $K_{k+1}$ in a graph $\tilde{G}_i$ is highly structured. For the following proof we set $C:= 2^{100}$.
	
	\begin{proof}[Proof of \Cref{thm: MinDegGen}]
		By \Cref{obs: ColorPatter} it suffices to find a $K_{k+1}$\textit{-free} color pattern $G_1,\dots,G_r$ on a vertex set $L$ of size $n\leq C^4k^2r^{2+\frac{30}{k}}\log^{20}r\log^{20}k$ such that $\alpha_k(G_i) < n/r$ for all $i \in [r]$. To that end, we first use Chebyshev's theorem to find a prime $q$ such that 
		\[\frac{C}{2}k^{\frac{1}{2}}r^{\frac{1}{2}+ \frac{15}{2k}}\log^{5} r \log^{5}k\leq q \leq Ck^{\frac{1}{2}}r^{\frac{1}{2}+ \frac{15}{2k}}\log^{5} r \log^{5}k.\]
		We quickly note the following implication which will be relevant later in the proof: \[\log q \leq \frac{1}{2} \log k + \frac{k+15}{2k} \log r + \mathcal{O}(\log\log r) \leq \log r \log k,\]
		for $r$ and $k$ sufficiently large.
		We choose $c= \lceil \frac{8r}{q}\rceil$ and note that indeed $c \leq \frac{q}{48 \log q}$, and the condition for Lemma \ref{lem:coloringpoints} is satisfied. Since $\lceil \frac{8r}{q}\rceil \lfloor \frac{q}{2}\rfloor \geq r$, applying \Cref{lem: graphlemma} gives us $r$ graphs $\tilde{G}_1,\dots,\tilde{G}_r$ on the vertex set $L$ where $q^4/2\leq |L|\leq q^4$. 
		Our job is done once we find inside every $\tilde{G}_i$ a $K_{k+1}\textit{-free}$ spanning subgraph $G_i$ satisfying $\alpha_k(G_i)<\frac{|L|}{r}$. Without loss of generality we fix the subscript $i$ to $1$ and probabilistically prove that such a subgraph $G_1 \subset \tilde{G}_1$ exists. We fix $\alpha:= r^{-\frac{15}{2k} }\log^{-4} r \log^{-4} k$ and carry out the following random procedure: We vertex partition each point-clique of $\tilde{G}_1$ independently into $k+1$ parts $R_0,\dots,R_k$ where every vertex  is independently placed in $R_j$ with probability $\frac{\alpha}{k}$ for $j \in [k]$ and in $R_0$ with probability $1-\alpha$. We then keep an edge $ab$ in $\tilde{G}_1$ if and only if there is some distinct $i,j \in [k]$ such that $a \in R_i$ and $b \in R_j$ in the partition of the unique point-clique containing $a$ and $b$. Each vertex is placed independently from other vertices and the partitions among the cliques are mutually independent. Let $\hat{G}$ be the probability space corresponding to our random procedure. We make the following two claims whose proofs will be slightly postponed:
		
		\begin{claim}\label{claim: ManyKk}
			\begin{equation}
				\mathbb{P}\Biggl( \exists A \in \binom{L}{|L|/r}: K_k \nsubseteq\hat{G}[A] \Biggr) < \frac{1}{2}.
			\end{equation}
		\end{claim}
		\begin{claim}\label{claim: NoKk+1}
			\begin{equation}
				\mathbb{P} \Biggl( K_{k+1}  \subseteq \hat{G}\Biggr) < \frac{1}{2}. 
			\end{equation}
			
		\end{claim}
		Given the above two claims, we choose for every $i \in [r]$, a $K_{k+1}$\textit{-free} spanning subgraph $G_i \subseteq \tilde{G}_i$ such that $\alpha_k(G_i)<\frac{|L|}{r}$. $G_1,\dots,G_r$ is then the desired color pattern. This proves \Cref{thm: MinDegGen} pending the proofs of the claims which we do next.

		\begin{proof}[Proof of Claim \ref{claim: ManyKk}]
			Fix some $A \in \binom{L}{|L|/r}$ and let $E_A$ be the event that $K_k \nsubseteq\hat{G}[A]$. Given a point-clique $K$, we say that $K$ is $A$\textit{-proper} if 
			$A_K:= A\cap K$ has size at least $\frac{q^2}{2^4r}=:t$. Furthermore, let us denote $A$\textit{-proper} point-cliques by $P_A$. By a simple double counting we get that:
			\begin{equation}
				\underset{K \in P_A}{\sum}|A_K| \geq \frac{|L|q}{2cr}- \frac{2q^3 t}{c}\geq \frac{|L|q}{4cr}.
			\end{equation}
			%\red{Is the 1/4 in the end correct?}
            %\textcolor{blue}{For this to follow we just need that $\frac{|L|q}{2cr} \geq 2\frac{2q^3t}{c}$ and since $|L|\geq \frac{q^4}{2}$ and $t=\frac{q^2}{2^4r}$ we should be good.} \red{Cool, probably I messed up the powers of 2.}
			For every $A$\textit{-proper} point-clique $K$, let us further vertex partition $A_K$ into $\lfloor \frac{|A_K|}{t}\rfloor=:s_K$ sets $K^1,\dots,K^{s_K}$ each of size $t$. We note that 
			\[E_A\subseteq\{K_k \nsubseteq \hat{G}[K^1],\dots,\hat{G}[K^{s_K}]\} \]
			and that the events $\{K_k\nsubseteq \hat{G}[K^i]\}_{i \in [s_K]}$ are mutually independent. Therefore, we have
			\begin{equation*}
				\mathbb{P}( E_A) = \underset{K \in P_A}{\prod}\underset{j \in [s_K]}{\prod}\mathbb{P}\Bigl(K_k\nsubseteq\hat{G}[K^j]\Bigr) \leq  \underset{K \in P_A}{\prod}\underset{j \in [s_K]}{\prod}k\Bigl(1-\frac{\alpha}{k}\Bigr)^t \leq \exp(-\underset{K \in P_A}\sum\frac{\alpha t}{2k} s_K) \leq \exp\Bigl(-\frac{\alpha }{4k}\underset{K \in P_A}\sum|A_K|\Bigr).
			\end{equation*}
			Where the first inequality follows by the union bound, the second by using $e^{-x}\geq 1-x$ for $x \in (0,1)$ and that $\frac{\alpha t}{k}\geq 2 \log k$, and the last by noting that $\lfloor \frac{|A_K|}{t}\rfloor \geq \frac{|A_K|}{2t}$.
			Using the union bound we can conclude:
			\begin{equation*}
				\mathbb{P}\Bigl( \underset{A \in \binom{L}{|L|/r}}{ \cup}E_A\Bigr)\leq \binom{|L|}{|L|/r} \exp \Bigl( -\frac{\alpha}{4k} \underset{K \in P_A}{\sum}|A_K| \Bigr) \leq \exp \Bigl( \frac{|L| \ln (er)}{r}-\frac{\alpha |L|q}{16kcr}\Bigr) < \frac{1}{2}.
			\end{equation*}
			The last inequality follows since 
            \begin{equation}\label{eq: techKk}
                \alpha q \geq 32kc\ln(er),
            \end{equation}

			which follows by our choice of $q$ and since $k < r \log^2 r$.
			%\textcolor{blue}{Please decide if this elaboration is needed:\\
            %\underline{Case 1: $c>1$} then $c= \lceil \frac{8r}{q}\rceil \leq \frac{16r}{q} $ and \Cref{eq: techKk} 
            % follows since $q^2\geq 2^9 kr \log^4k \log^4r \ln(er)$ by our choice of $q$. \\
            % \underline{Case 2: $c= 1$}. \Cref{eq: techKk} becomes $q\geq 2^5k \log^4k \log^4r\ln(er)$ which is true by our choice of $q$ and the inequality $k \leq r \log^2 r.$ 
             %}
             %\red{Is the $r^{\frac{15}{2k}}$ from $\alpha$ missing? Anyway, I had no issue with this estimation, so the clarification is not necessary as far as I am concerned.}
		\end{proof}
		\begin{proof}[Proof of Claim \ref{claim: NoKk+1}] Recall that there are only two possible types of $K_{k+1}$ in $\tilde{G}_1$: Degenerate $K_{k+1}$'s which get deleted by the Turánization of the point-cliques; and $K_{k+1}$'s that correspond to $(k+1)$\textit{-fans}. If $P_1$ is the set of point-cliques in $\tilde{G}_1$, then there are at most 
			\[|P_1|.|L| \binom{2q/c}{k} \leq q^7 \binom{2q/c}{k}\]
			$(k+1)$\textit{-fans} in $\tilde{G}_1$.
			Let $F$ be a $(k+1)$\textit{-fan} in $\tilde{G}_1$, and let $K_0,\dots,K_k$ be the point-cliques containing at least two vertices  of $F$, which we shall call relevant. Without loss of generality $|K_0\cap F|=k$  while $|K_i\cap F|=2$ for every $i\in [k]$. The probability that $\hat{G}[F]\cong K_{k+1}$ is at most $\alpha^{3k}$ since every vertex in $F$ must be thrown in the "active portion" of the partition of every relevant point-clique containing it. Therefore, using the union bound again we conclude that:
			\[\mathbb{P}\Bigl(K_{k+1} \subset \hat{G} \Bigr) \leq q^7 \binom{2q/c}{k}\Bigl(\frac{1}{r^{45/2k}\log^{12}k \log^{12}r}\Bigr)^k\leq \exp \Bigl(7\ln q-k\ln\log q -7.4 \ln r\Bigr)<\frac{1}{2}.\]
			for large enough $k$. 
            %\textcolor{blue}{Please decide if the following elaboration is needed.
            The second inequality follows since 
            \[\binom{2q/c}{k} \Bigl(\frac{1}{r^{\frac{45}{2k}}\log^{12}k\log^{12}r}\Bigr)^k \leq \Bigl(\frac{2qe}{ck} \Bigr)^k\Bigl(\frac{1}{r^{\frac{45}{2k}}\log^{12}k\log^{12}r}\Bigr)^k \leq \Bigl(\frac{1}{r^{\frac{15}{2k}}\log r \log k}\Bigr)^k \leq \Bigl(\frac{1}{r^{7.5/k} \log q}\Bigr)^k.\]
            %}
            %\red{This elaboration I do like. However, where is the $C^2$ coming from? Also how does the 7.5 change to a 7.4?
            %}
		\end{proof}

        This completes the proofs of Claims \ref{claim: ManyKk} and \ref{claim: NoKk+1} and hence the proof of \Cref{thm: MinDegGen}.
	\end{proof}

	\section{Lower Bound}\label{sec:lowerbounds}
	While the best (partial) upper bounds for the minimum degree of minimal Ramsey graphs are quadratic in both parameters (suppressing logarithmic factors), we do not currently have a matching lower bound. We know that $\Omega(rk^2)=\mathrm{ssat}_r(K_k)\leq s_r(K_{k})$ where the equality follows by a theorem of Damásdi et al.\ \cite{damasdi2021saturation} and the inequality by an observation of Tran \cite{tran2022two}. In this section we provide another lower bound quadratic in $r$.

	\begin{reptheorem}{thm: LowerKk+1}
		For all $k,r \geq 3$, $s_r(K_{k+1}) \geq \frac{kr^2}{16}$.
	\end{reptheorem}
	Let us first prove the following lemma 
	\begin{lemma}\label{lemma: ErdosRogersK}
		Let $k\geq 3$ and let $G$ be a $K_{k+1}$\textit{-free} graph on $n \in \mathbb{N}_{+}$ vertices. There exists a vertex subset $V' \subset V(G)$ such that $|V'| \geq \frac{1}{2}\sqrt{kn}$ and $G[V']$ is $K_{k}$\textit{-free}.
	\end{lemma}
	\begin{proof}
		If a $K_{k+1}$-free graph $G$ has maximum degree $d$, then 
		the neighborhoods of vertices are $K_k$-free. Zykov~\cite{Zykov} proved that the maximum number of copies of $K_{k - 1}$ in a $K_k$-free 
		graph with $d$ vertices is at most $D = (d/(k - 1))^{k - 1}$, with equality achieved by a balanced complete $(k - 1)$-partite graph. So the number of $K_k$ 
		in $G$ is at most $nD/k$. If we randomly sample vertices of the graph with probability $p = D^{-1/(k - 1)}$, then 
		the expected number of vertices remaining after we remove one vertex from each copy of $K_k$ is at least 
		\[ pn - p^k\frac{nD}{k} \geq \Bigl(1 - \frac{1}{k}\Bigr) \frac{n}{D^{1/(k - 1)}} \geq \frac{(k - 1)^2 n}{kd} \geq \frac{kn}{4d}.\] 
		We conclude that $G$ contains a $K_k$-free subgraph $H$ where
		\[ |V(H)| \geq \max\Bigl\{d,\frac{kn}{4d}\Bigr\} \geq \frac{1}{2}\sqrt{kn}.\]
	\end{proof}
	\begin{proof}[Proof of Theorem \ref{thm: LowerKk+1}]
		First, let us note the following recursion for all $r,k\geq 3$.  \begin{equation}\label{recursion}
			P_r(k) \geq P_{r-1}(k)+ \Bigg\lceil \frac{1}{2}\sqrt{kP_{r}(k)}\Bigg\rceil.
		\end{equation}
		
		Indeed, Let $G_1,\dots,G_r$ be an optimal $K_{k+1}$\textit{-free} color pattern on the vertex set $V$ such that every $[r]$\textit{-coloring} of $V$ contains a strongly monochromatic $K_k$. This means $|V|=P_r(k)$ and we can use Lemma \ref{lemma: ErdosRogersK} applied on the $K_{k+1}$\textit{-free} graph $G_r$ to find a subset $V'\subset V$ of size at least $\frac{1}{2}\sqrt{k P_r(k)}$ such that $G_r[V']$ is $K_{k}$\textit{-free}. Let us color $V'$ with the color $r$ and note that every extension of this coloring using colors from $[r-1]$ contains a strongly monochromatic $K_k$ in $V-V'$ in one of the colors $[r-1]$ and therefore, $|V-V'| \geq P_{r-1}(k)$ proving our claim. Next, we claim that for $k,r \geq 3$, we have $P_r(k) \geq \frac{kr^2}{16}$. Note that this claim is equivalent to the statement of \Cref{thm: LowerKk+1} by \Cref{lem: ColorPattern}. We proceed by induction on $r$. Fix some $k \geq 3$, for the base case of our induction we have $r=3$. However, $P_3(k) \geq P_2(k)=k^2\geq \frac{3^2k}{16}.$ Next, we employ the recursion in \ref{recursion} to have \begin{equation}
			P_r(k)-\frac{1}{2}\sqrt{kP_r(k)} \geq P_{r-1}(k) \geq \frac{k(r-1)^2}{16}, 
		\end{equation}
		where the last inequality follows by the induction hypothesis. Let $x:= \sqrt{P_r(k)}$, then we know that
		\[x^2-\frac{1}{2}\sqrt{k}x - \frac{k(r-1)^2}{16} \geq 0\]
		or equivalently 
		\[\Biggl(x-\frac{\sqrt{k}}{4} -\sqrt{\frac{k}{16}+\frac{k(r-1)^2}{16}}\Biggr) \Biggl(x-\frac{\sqrt{k}}{4} +\sqrt{\frac{k}{16}+\frac{k(r-1)^2}{16}} \Biggr)\geq 0.\]
		Then either $\sqrt{P_r(k)} =x\leq \frac{\sqrt{k}}{4}-\sqrt{\frac{k}{16}+\frac{k(r-1)^2}{16}}<0$
		which gives us a contradiction; or
		\[\sqrt{P_r(k)}=x \geq \frac{\sqrt{k}}{4} + \sqrt{\frac{k}{16}+\frac{k(r-1)^2}{16}} \geq \sqrt{\frac{kr^2}{16}}\]
		and we are done. 
		%\red{We might have to add another power of 2 in the denominator since $(r-1)^2 \geq r^2$ will not be true ($r$ could be huge and $k$ very small in the above).}
        %\textcolor{blue}{I still think it is correct. the difference between $r$ and $\sqrt{(r-1)^2}$ is approaching a constant as $r$ goes to infinity. This constant (multiplied by a k ) will be accounted for by the other terms in $k$. Here is a formal argument for it: Equivalently by squaring and dividing by $16$, we want to show that 
        %\[2k+ k(r-1)^2+2k\sqrt{1+(r-1)^2}\geq kr^2\] or
        %\[2+ (r-1)^2+2\sqrt{1+(r-1)^2}\geq r^2\] or equivalently 
        %\[r\leq 3/2+\sqrt{1+(r-1)^2}\] and this is obviously true.}
        %\red{Ok cool, I did not take the whole thing into account. I wonder if the average reader will, but we can leave it as is.}
	\end{proof}

	\section{Semisaturated Ramsey Numbers}\label{section: Tran}
	In this short section we prove Theorem~\ref{thm:semisaturation} about our upper bound on semisaturated Ramsey numbers. 
	\begin{proof} For our proof we must construct an edge $r$-coloring of $K_n$ with $n=4(k-1)^2r^2$ such that any extension of it to an $r$-edge coloring of $K_{n+1}$ creates a new monochromatic $K_{k+1}$. 

		By Chebyshev's theorem, we can find a prime $q$ such that $(k-1)r < q < 2(k-1)r$.         
        Consider the affine plane $(\cP,\cL)$ of order $q$ and denote its $q+1$ parallel classes of lines by $\cL_1,\dots,\cL_{q+1}$. 
        Define first an edge $(q+1)$-coloring of the complete graph $K_{q^2}$
        with vertex set $\cP$ as follows: 
   %     Define a graph $G \in \cG_{q+1}$ on the points of the affine plane as follows: 
   the edge between two vertices receives color $i$ if and only if the line they span on the affine plane is in $\cL_i$. One can observe that every color class is then the union of $q$ disjoint cliques of order $q$, and any two cliques of a different color meet in exactly one vertex. To create an $r$-coloring we can for example collapse the color classes between $r$ and $q+1$ into one.\\
	We need to show that adding a vertex and $r$-coloring the edges incident to it necessarily creates a new monochromatic $K_{k+1}$. There is certainly a color, say color $i \in [r]$, that occurs at least the average number, i.e. $q^2/r > (k-1)q$ times among the $q^2$ edges incident to the new vertex. Since there are $q$ pairwise disjoint monochromatic cliques in color $i$ covering the whole vertex set, one of these cliques must have at least $k$ vertices which are connected to the new vertex via an edge of color $i$. These vertices, together with the new vertex, form a new $K_{k+1}$ that is monochromatic in color $i$.    
 %   By passing to the largest color class in the neighborhood of the added vertex, it suffices to show that for all $i \in [q+1]$ every set of size at least $|\cP|/r$ induces a copy of $K_k$ in color $i$. This will prove the theorem,
%		as the number of vertices of $G$ is $q^2 < 4(k-1)^2r^2 < 4k^2r^2$ and one can easily recolor the graph $G$ to obtain a member $G' \in \cG_r$ by taking unions of color classes. Formally: for any surjective function $f:[q+1]\to[r]$, we can define a graph $G' \in \cG_r$ by defining $G_j' = \cup_i G_i$, where the index $i$ runs over all $i \in [q+1]$ such that $f(i) = j$, and $G_j'$ (resp.\ $G_i$) denotes the subgraph of $G'$ (resp.\ of $G$) induced by color $j$ (resp. $i$).   \\
%		
%		So fix a color $i \in [q+1]$ and consider a set of points $U$ of size at least $|\cP|/r = q^2/r > (k-1)q$. Since there are exactly $q$ lines in a parallel class, we find by pigeonhole principle that there must be a line $\ell \in \cL_i$ containing $k$ points of $U$. By definition of $G$, this corresponds to a copy of $K_k$ in color $i$.
	\end{proof}
	
	\section{Concluding remarks} \label{sec:remarks}
	
	$\bullet$ {\bf Asymptotics of $s_r(K_{k})$.} 
	Tantalizing problems remain open in all ranges of the parameters. 
	When $k$ and $r=r(k)$ both tend to infinity, the main question concerns correct exponents of $k$ and $r$. 
	If $k\leq r \log^2 r$, Theorems~\ref{thm: MinDegGen} and \ref{thm: LowerKk+1} give 
	$$ \frac{1}{16}r^2 k \leq s_r(K_{k})  \leq (rk)^{2+o(1)},$$
	so the exponent of $k$ is waiting to be settled.
	If $k > r \log^2 r$, then the bound of H\`an et al.\ ~\cite{han} gives 
	$$r k^2(1+o(1)) \leq s_r(K_{k})  \leq r^{3+o(1)} k^{2+o(1)},$$
	which leaves the exponent of $r$ up for grabs.
	
	For constant $k\geq 4$, Theorem~\ref{thm: MinDegConst} and the lower bound (\ref{eq:fox-constantk}) of Fox et al.\ ~\cite{fox2016minimum} gives
	$$ c_k r^2 \frac{\log r}{\log\log r} \leq s_r(K_{k})  \leq C_k r^2 \log^2 r,$$
	so the status of a factor of $\log r \log \log r$ remains unsettled. 
	For constant $r\geq 3$, the bounds 
	$$ c_r k^2 \leq s_r(K_{k})  \leq C_rk^2 \log^2 k.$$ of H\`an et al.\ \cite{han} are leaving us with the challenge whether how much of the factor $\log^2 k$ is necessary. 
	Making progress on any of these bounds would be very interesting. 
	We are especially eager to discover new avenues to obtain lower bounds for the problem. 
	
	For completeness we recall that for $r=2$ colors the exact value $s_2(k)=(k-1)^2$ was determined by Burr et al.\ ~\cite{burr1976graphs} and for $k=3$ the order of magnitude $s_r(3) = \Theta (r^2\log r)$ is known by Guo and Warnke~\cite{guo2020packing} and Fox et al.\ \cite{fox2016minimum}. 
    The determination of the constant factor in the latter problem is related to the analogous question for the Ramsey number $R(3,\ell)$, a notorious open problem.

	$\bullet$ {\bf Separation in the (semi)saturation problem.} 

By the definition of the $r$-color Ramsey number $R_r(K_{k+1})$, there exists an edge $r$-coloring of the complete graph on $R_r(K_{k+1})-1$ vertices, which does not contain any monochromatic $K_{k+1}$. Moreover, $R_r(K_{k+1}) -1$ is the largest number of vertices on which such an edge $r$-coloring exists and hence this coloring is in the family ${\cal RC}_r(K_{k+1})$.  
%	Let ${\cal RC}_r(K_k)$ be the set of $r$-colorings of some complete graph, such that any extension of it to an $r$-coloring of the edges of the complete graph with one more vertex creates a new monochromatic copy of $K_k$. The motivation behind this notion is that the classical $r$-color Ramsey number $R_r(K_k)$ is exactly $1$ plus the {\em largest} integer $n$ such that there exists an $r$-coloring of $E(K_n)$ in ${\cal RC}_r(K_k)$ without a monochromatic $K_k$. 
	Dam\'asdi et al.\ ~\cite{damasdi2021saturation} defined $\mathrm{sat}_r(K_{k+1})$ to be the {\em smallest} integer $n$ such that there exists an $r$-coloring of $E(K_n)$ in ${\cal RC}_r(K_{k+1})$ without any monochromatic $K_{k+1}$. Thus $\mathrm{sat}_r(K_{k+1}) < R_r(K_{k+1})$. 
	This concept was inspired by the definition of Erd\H os, Hajnal and Moon~\cite{erdos1964problem} of the saturation number $\mathrm{sat}(n,H)$ of graph $H$ opposite of its Tur\'an number $\mathrm{ex}(n,H)$. 
 
 Then trivially $\mathrm{ssat}_r(K_{k+1}) \leq \mathrm{sat}_r(K_{k+1})$, as in the definition of $\mathrm{ssat}_r(K_{k+1})$ we are allowed to consider {\em all} members of ${\cal RC}_r(K_{k+1})$, not only those without monochromatic $K_{k+1}$. We can also easily see that our main focus, the smallest minimum degree $s_r(K_{k+1}) = P_r(k)$ is sandwiched between these two saturation parameters. 
	
	\begin{proposition} \label{prop:ssat-P-sat}
		$\mathrm{ssat}_r(K_{k+1}) \leq P_r(k) \leq \mathrm{sat}_r(K_{k+1})$
	\end{proposition}
	\begin{proof}
		 For the second inequality, we connect the nomenclature of $P_r(k)$ to the color saturation parameter $\mathrm{sat}_r(K_{k+1})$. Firstly, an edge $r$-coloring of $K_n$ can be identified with an $r$-color pattern $G_1, \ldots , G_r$ on $[n]$, such that $\cup_{i=1}^r E(G_i) = E(K_n)$. Moreover, an edge $r$-coloring having the property that every extension of it to $n+1$ vertices creates a new monochromatic $K_{k+1}$ is equivalent to the corresponding $r$-color pattern $G_1, \ldots , G_r$ having the property that every $r$-coloring of $[n]$ contains a strongly monochromatic $K_k$. Indeed, extensions of an edge $r$-coloring of $K_n$ to the edges incident to the new vertex $(n+1)$ are in one-to-one correspondence with the vertex $r$-colorings of $[n]$ (namely, the color of an extension edge is just the color of the endpoint of that edge in $[n]$ in the vertex coloring). Then having a ''new'' monochromatic $K_{k+1}$ in the extension is equivalent to having a strongly monochromatic $K_k$ in the vertex coloring. 
		The second inequality of the proposition then follows because for the definition of $P_r(k)$ we do {\em not} require the $r$-color pattern to partition the whole edge set of the clique, while for $\mathrm{sat}_r(K_{k+1})$ we do, so the minimum $n$ for the latter is taken over a subset of the set we take the minimum of for the former. 
        
        For the first inequality, which has already been observed by Tran~\cite{tran2022two}, one can take an $r$-color pattern $G_1, \ldots , G_r$ on the optimal number $n=P_r(k)$ of vertices and color arbitrarily the uncolored edges in $E(K_n) \setminus \cup E(G_i)$ by $r$ colors. This provides an $r$-coloring of $E(K_n)$ that is in ${\cal RC}_r(K_{k+1})$. 
     \end{proof}

It is natural to wonder how tight the two inequalities of the previous proposition are in the various ranges of the parameters. 

For $r=2$ for example it is known \cite{damasdi2021saturation} that all three functions are equal to $k^2$. 
    %both inequalities of the previous proposition are equalities.
On the other hand for any constant $k\geq 2$ our Theorem~\ref{thm:semisaturation} and the lower bounds of \cite{fox2016minimum} do separate the order of magnitude of $\mathrm{ssat}_r(K_{k+1})$ and $P_r(k)$ as $r$ tends to infinity. 

We believe that this to be true for any other ranges of the parameters.

\begin{conjecture} \label{con:ssat-P} For any $r=r(k) \geq 3$, as $k$ tends to infinity we have 
		$\mathrm{ssat}_r(K_{k+1})\ll P_r(k)$. 
        \end{conjecture}
    
   	More modestly, it would even be interesting to decide whether there is significant separation between $\mathrm{ssat}_r(K_{k+1})$ and $\mathrm{sat}_r(K_{k+1})$.  
 
For $r=3$ H\`an, R\"odl and Szab\'o~\cite[Conjecture 2]{han}  conjectured an asymptotic separation between the three- and two-color case of the smallest minimum degree parameter. This is plausible, yet we do not even know whether there is {\em any} $r$ for which 
$s_r(K_{k+1}) = P_r(k) \gg s_2(K_{k+1}) = k^2$ as $k \to \infty$. 
Since $P_r(k)$ is non-decreasing in $r$ and by Theorem~\ref{thm:semisaturation} the order of $\mathrm{ssat}_r(K_{k+1})$ is quadratic in $k$ for every fixed $r$, the conjecture of H\`an et al.\ would also imply Conjecture~\ref{con:ssat-P} for any fixed $r\geq 3$.

    Concerning the separation in the second inequality of Proposition~\ref{prop:ssat-P-sat} we are less convinced. For the case of $k=2$ we tend to agree with authors of \cite{guo2020packing} who believe that their construction could be improved so that all (and not only a $(1-\varepsilon)$ proportion) of the edges of the complete graph are covered by some $K_3$-free $r$-color pattern on $\Theta ( r^2\log r)$ vertices. 

\begin{conjecture}
    $P_r(2) = \Theta (\mathrm{sat}_r(K_{3})).$ 
    \end{conjecture}

	$\bullet$ {\bf Monotonicity of $s_r(K_k)$ in $k$.} Finally, we would like to reiterate the humbling conjecture of Fox et al.~\cite[Conjecture 5.1]{fox2016minimum} stating that $s_r(K_{k+1}) \geq s_r(K_k)$ for every $r\geq 2$ and $k\geq 2$. Recall that from the identity $s_r(K_{k+1})=P_r(k)$, it is not difficult to see that $s_r(K_{k+1})$ is non-decreasing in $r$.

	%$\bullet$ {\bf Lower bounds on the Erd\H{o}s-Rogers function.} Given two positive integers $k$ and $n$, the Erd\H{o}s-Rogers function $f_k(n)$ is the largest integer $m$ such that every $K_{k+1}$\textit{-free} graph on $n$ vertices contains a $K_k$\textit{-free} vertex subset of size $m$. Lower bounds on $f_k(n)$ lead to lower bounds on $P_r(k)$ through the following recursion :
	%\[P_{r}(k) \geq P_{r-1}(k)+ f_k(P_r(k)).\]
	%Indeed, this approach was used in \Cref{sec:lowerbounds}. It is of interest then to provide lower bounds on the Erd\H{o}s-Rogers function $f_k(n)$. The best known lower bound is due to Dudek and Mubayi and reads as follows: 
	%\[f_k(n)= \Omega\Bigl(\frac{\sqrt{n\log n} }{\log \log n}\Bigr).\]
	%The proof proceeds as follows: A $K_{k+1}$ graph $G$ with maximum degree $\Delta$ contains a $K_k$\textit{-free} subset of size $\Delta$ and by a result of Shearer ~\cite{shearer1995independence} an independent set of size $\Omega \bigl( \frac{n\log \Delta}{ \Delta \log \log \Delta }\bigr)$. This $1/\log \log$ factor in the lower bound of $f_k(n)$ is inherited from Shearer's result. Shearer's result however gives a large independent set, a much stronger condition than $K_k$\textit{-freeness}. Is it possible to provide better lower bounds on the Erd\H{o}s-Rogers function by appropriating Shearer's Lemma or abandoning it all together? \textcolor{blue}{I just realized that this has been asked/mentioned in the Mubayi-Verstraete paper.}\\

	\printbibliography

\end{document}